\documentclass{amsart}

\usepackage{amsmath}
\usepackage{amssymb}
\usepackage{mathrsfs}
\usepackage{amsthm}
\usepackage{comment}
\usepackage{color}

\title{Annihilating-Ideal Graph of $C(X)$}

\author[M. Badie]{Mehdi Badie}
\address{Department of Basic Sciences, Jundi-Shapur University of Technology, Dezful, Iran}
\email{badie@jsu.ac.ir}


\theoremstyle{plain}
\newtheorem{Thm}{Theorem}[section]
\newtheorem{Lem}[Thm]{Lemma}
\newtheorem{Def}[Thm]{Definition}
\newtheorem{Pro}[Thm]{Proposition}
\newtheorem{Cor}[Thm]{Corollary}
\newtheorem{Exa}[Thm]{Example}

\renewcommand{\O}{\mathbf{O}}
\newcommand{\I}{\mathbf{I}}
\newcommand{\AG}{\mathbb{AG}}
\newcommand{\DGX}{\mathbb{DG}(X)}
\newcommand{\AC}{\A{X}}
\newcommand{\AGC}{\mathbb{AG}(X)}
\newcommand{\A}[1]{\mathbb{A}(#1)^*}
\newcommand{\ff}{if and only if }

\newcommand{\An}{\mathrm{Ann}}

\newcommand{\Ze}{\mathrm{Z}}
\newcommand{\Co}{\mathrm{Coz}}
\newcommand{\Ec}{\mathrm{ecc}}
\newcommand{\Di}{\mathrm{diam}}
\newcommand{\Ra}{\mathrm{Rad}}
\newcommand{\Gi}{\mathrm{girth}}

\newcommand{\Cli}{\mathrm{clique}}
\newcommand{\gi}{\mathrm{gi}}
\newcommand{\dt}{\mathrm{dt}}
\newcommand{\Ge}[1]{\big< #1 \big>}

\newcommand{\Q}{\mathbb{Q}}
\newcommand{\R}{\mathbb{R}}

\begin{document}

\begin{abstract}
In this article we study the annihilating-ideal graph of the ring $C(X)$. We have tried to associate the graph properties of $\AGC$, the ring properties of $C(X)$ and the topological properties of $X$.  We have shown that $ X $ has an isolated point \ff $ \R $ is a direct summand of $ C(X) $ \ff $ \AGC $ is not triangulated. Radius, girth, dominating number and clique number of the $\AGC$ are investigated. We have proved that $ c(X) \leqslant \dt(\AGC) \leqslant w(X) $ and $ \Cli \AGC = \chi \AGC = c(X) $.
\end{abstract}

\maketitle

\section{Introduction}

Let $ G = \big< V(G) , E(G) \big> $ be an undirected graph. A vertex which adjacent to just one vertex is called a \emph{leaf vertex}.  The degree of a vertex of $ G $ is the number of edges incident to the vertex. If $ G $ has a vertex which adjacent to all another vertices, then $ G $ is called a \emph{star graph}.  For each vertices $ u $ and $ v $ in $ V(G) $, the length of the shortest path between $ u $ and $ v $, denoted by $ d(u,v) $, is called the \emph{distance} between $ u $ and $ v $. The \emph{diameter} of $ G $ is defined $ \sup \{ d(u,v) : u,v \in V(G) \} $, the diameter of $ G $ is denoted by $ \Di(G) $. The \emph{eccentricity} of a vertex $ u $ of $ G $ is denoted by $ \Ec(u) $ and is defined to be maximum of $ \{ d(u,v) : u \in G \} $. The minimum of $ \{ \Ec(u) : u \in G \} $, denoted by $ \Ra(G) $, is called the \emph{radius} of $G$. For every $ u,v \in V(G) $, we denote the length of the shortest cycle containing $ u $ and $ v $ by $ \gi(u,v) $ and the minimum length of cycles in $ G $, is denoted by $ \Gi(G) $ and is called the \emph{girth} of graph, so $ \Gi(G) = \min \{ \gi(u,v) : u , v \in V(G) \} $. We say $ G $ is \emph{triangulated} (\emph{hypertriangulated}) if each vertex (edge) of $ G $ is vertex (edge) of some triangle. A subset $ D $ of $ V(G) $ is called \emph{dominating set} if for each $ u \in V(G) \setminus D $, there is some $ v \in D $, such that $ v $ is adjacent to $ u $. The \emph{dominating number} of $ G $, denoted by $ \dt(G) $, is the smallest cardinal number of dominating sets of $ G $. We say two vertices $ u $ and $ v $ are \emph{orthogonal} and denote by $ u \perp v $, if $ u $ and $ v $ are adjacent and there are no any vertex which adjacent to both vertices $ u $ and $ v $. If for every $ u \in V(G) $, there is some $ v \in V(G) $ such that $ u \perp v $, then $ G $ is called \emph{complemented}. A complete subgraph of $ G $ is called a \emph{clique} of $ G $. The supremum of the cardinality of cliques of $ G $, denoted by $ \Cli(G) $, is said the \emph{clique number} of $ G $. The \emph{chromatic number} of $ G $, denoted by $ \chi(G) $, is the minimum cardinal number of colors needed to color vertices of $ G $ so that no two vertices have that same color. Clearly, $ \Cli(G) \leqslant \chi(G) $. 

Throughout the paper $ R $ is denoted the commutative ring with unity. For each ideal $I$ of $ R $ and each element $ a $ of $ R $, we denote the ideal $ \{ x \in R : ax \in I\}$ by $ ( I : a ) $. When $ I = \big< 0 \big>  $ we write $ \An(a) $ instead of $\left(\big<0\big>:a\right)$ and call this the \emph{annihilator} of $a$. If for each subset $ S $ of $ R $ there is some $ a \in R $ such that $ \An(S) = \An(a) $, then we say $ R $ is satisfy \emph{infinite annihilating condition} ($ R $ is an i.a.c ring). We denote the family of all non-zero ideal with non-zero annihilating by $ \A{R} $. $ \AG(R) $ is a graph with vertices $ \A{R} $ and two distinct vertices $ I $ and $ J $ are adjacent, if $ IJ = \{ 0 \} $.  

In this paper $ C(X) $ is denote the set of all real-valued continuous function on a Tychonoff space $ X $. The \emph{weight} of $X$, denoted by $w(X)$, is the infimum of the cardinalities of bases of $X$. The \emph{cellularity} of $X$, denoted by $c(X)$, is defined 
\[ \sup\{|\mathcal{U}|:{\mathcal U} \text{ is a family of mutually disjoint nonempty open subsets of  } X \}\]
For any $ f \in C(X)$, we denote $f^{-1} \{ 0 \}$   and $ X \setminus f^{-1}  \{0\} $ by $Z(f )$ and $\Co(f )$,  respectively. Every set of the form  $Z(f)$($\Co(f)$) is called \emph{zeroset} (\emph{cozeroset}). An ideal $ I $ of $ C(X) $ is called \emph{fixed} (\emph{free}) if $ \bigcap_{f \in I} Z(f) \neq \emptyset $ ($ \bigcap_{f \in I} Z(f) = \emptyset $). Suppose that $ A \subseteq X $, We denote $ \{ f \in C(X) : A \subseteq Z(f) \} $ and $ \{ f \in C(X) : A \subseteq Z(f)^\circ \} $ by $ M_A $ and $ O_A $, respectively. When $ A = \{p\} $ we write $ M_p $ and $ O_p $ instead of $ M_{\{p\}} $ and $ O_{\{p\}} $, respectively, it is clear that $ M_A = \bigcap_{p \in A} M_p $ and $ O_A = \bigcap_{p \in A} O_p $. By \cite[Theroem 7.3(Gelfand-Kolmogoroff)]{gillman1960rings}, $ \{ M^p : p \in \beta X \} $ is the family of all maximal ideal of $C(X)$. An ideal $ I $ of $ C(X) $ is called $ z $-ideal, if $ Z(f) = Z(g) $ and $ f \in I $, then $ g \in I $. Clearly, for every ideal $ I $ of $ C(X) $, $ Z^{-1}(Z(I)) $ is a the smallest $ z $-ideal containing $ I $. For more details we refer the reader to  \cite{bondy1976m} \cite{gillman1960rings}, \cite{atiyah1969introduction} and \cite{stephen1970general}.

The studding of some graphs on $ C(X) $ is an interesting. In these investigations  were tried to associate the ring properties of $C(X)$, the graph properties of graphs on $C(X)$ and the topological properties of $X$. In \cite{azarpanah2005zero,amini2011graph,badie2016comaximal} the zero-divisor graph, the comaximal ideal graph of $C(X)$ and comaximal graph of $ C(X) $ were studied. In \cite{behboodi2011annihilating,behboodi2011annihilating2}, the studying annihilating-ideal graph of commutative rings were started. On later, this investigation was continued in many papers, for instance see \cite{aalipour2012coloring,nikandish2014dominating,nikandish2015domination,chelvam2015connectivity,aliniaeifard2015annihilating,guo2017rings,selvakumar2018crosscap}. 

In this article we study the annihilating-ideal graph of $ C(X) $. We abbreviate $ \A{C(X)} $ and $ \AG(C(X)) $ by $ \AC $ and $ \AGC $, respectively. If $ | X  | = 1 $, then $ \AC = \emptyset $, so we assume $ | X | > 1 $, throughout the paper.

In the reminder of this section we give some propositions which were concluded immediately from the native algebraic properties of $ C(X) $ and \cite{behboodi2011annihilating,behboodi2011annihilating2,azarpanah2005zero}. In Section 2, we define maps $ \O $ from the family of all subsets of $ C(X) $ onto the family of all open subsets of $ X $ and $ \I $ from  the family of all subset of $ X $ into the family of all ideals of $ C(X) $. We study these maps and use of these notions to study the graph. We show that $ I $ is adjacent of $ J $ \ff $ \O(I) \cap \O(J) = \emptyset $, the non zero ideal $ I $ is zero divisor \ff $ \overline{\O(I)} \neq X $, $I(U) \in \AC $ \ff $ \overline{\O(I)}^{^\circ} \neq \emptyset $. In Section 3, we investigate the radius of the graph and we show that $ \AGC $ is star \ff $ |X| = 2 $. Section 4, is devoted to the girth of the graph. In this section we show that if $ |X| > 2 $, then $ \Gi \AGC = 3 $. Also, we show that an ideal in $ \AC $ is a leaf vertex \ff $ X \setminus \overline{\O(I)} = X $.  The studying the dominating number of the graph is the subject of Section 5. In this section we show that the clique number and chromatic number of $ \AGC $, and cellularity of $ X $ are equal.  

\begin{Pro}
	The following statements are equivalent.
	\begin{itemize}
		\item[(a)] $ |X| = 2 $.
		\item[(b)] $ \Di(\AGC) = 1 $.
		\item[(c)] $ \Cli \AGC = 2 $.
		\item[(d)] $ \AGC $ is a bipartite graph by two nonempty parts.
		\item[(e)] $ \AGC $ is a complete bipartite graph by two nonempty parts.
	\end{itemize}
\end{Pro}
\begin{proof}
	It is concludes from \cite[Theorem 1.4]{behboodi2011annihilating2} and \cite[Corollary 2.1]{nikandish2015domination}.
\end{proof}

\begin{Pro}
	The following statements hold.
	\begin{itemize}
		\item[(a)] $ X $ has at least 3 points \ff $ \Di(\AGC) = 3 $.
		\item[(b)] $ \chi(\AGC) = \Cli(\AGC) $.
	\end{itemize}
\label{diameter and clique=chi}
\end{Pro}
\begin{proof}	
	(a). It implies from \cite[Proposition 1.1]{behboodi2011annihilating2}, \cite[Corollary 1.3]{azarpanah2005zero} and  the previous proposition.
	
	(b). It is evident, by \cite[Corollary 2.11]{behboodi2011annihilating2}.
\end{proof}

The following proposition is an immediate consequence of \cite[Theorem 1.4]{behboodi2011annihilating}, \cite[Corollaries 2.11 and 2.12]{behboodi2011annihilating2} and this fact that $ X $ is finite \ff $ C(X) $ has just finitely many ideals.

\begin{Pro}
	The following statements are equivalent.
	\begin{itemize}
		\item[(a)] $ \AGC $ is a finite graph.
		\item[(b)] $ C(X) $ has only finitely many ideals.
		\item[(c)] Every vertex of $ \AGC $ has a finite degree.
		\item[(d)] $ X $ is finite.
		\item[(e)] $ \chi(\AGC) $ is finite.
		\item[(f)] $ \Cli(\AGC) $ is finite.
		\item[(g)] $ \AGC $ does not have an infinite clique.
		\item[(h)] $ \chi(\Gamma(C(X))) $ is finite.
	\end{itemize}
\end{Pro}

\section{$\I(U)$ and $\O(I)$}

For each subset $S$ of $C(X)$, we denote $ \bigcup_{f\in S} \Co(f) $ by $\O(S)$, and for each subset $َU$ of $X$, we denote $ \{ f \in C(X) : U \subseteq \Ze(f) \} = M_U = \bigcap_{a \in U} M_a $ by $ \I(U) $. It is clear that $ \O(S) = X \setminus \left( \bigcap_{f \in S} \Ze(f) \right)  $ and if $G$ is an open set in $X$, then $ \I(G) = O_G $. First, in this section we study the properties of these maps, then by these maps the edges and vertices of $ \AGC $ are investigated.

\begin{Lem}
	Let $S$ and $T$ be two subsets of $C(X)$, $f$ be an element of $C(X)$ and $U$ and $V$ be two subsets of $X$. The following hold.
	\begin{itemize}
		\item[(a)] If $ S \subseteq T $, then $\O(S) \subseteq \O(T)$.
		\item[(b)] If $ U \subseteq V $, then $\I(V) \subseteq \I(U)$.
		\item[(c)] $ \O(S) = \emptyset $ \ff $ S = \{ 0 \} $.
		\item[(d)] $ \O(S) = X $ \ff $ \Ge{S} $ is a free ideal.
		\item[(e)] $ \I(U) = \{ 0 \} $ \ff $ U $ is dense in $X$.
		\item[(f)] $ \I(U) = C(X) $ \ff $ U = \emptyset $.
		\item[(g)] $ \O(\Ge{f}) = \Co(f) $.
		\item[(h)] $ \I(U) = \I(\overline{U}) $.
	\end{itemize}
\label{order}
\end{Lem}
\begin{proof}
	It is straightforward.
\end{proof}

\begin{Pro}
	Let $S$ be a subset of $C(X)$. If $ I = \big< S \big> $, then $ \O(I) = \O(S) $.
\label{G generated}
\end{Pro}
\begin{proof}
	By Lemma \ref{order}, $ \O(S) \subseteq \O(I) $. Suppose that $ g \in I $, so a finite family $ \{ f_i \}_{i=1}^n $ of elements of $S$ and a finite family $\{ g_i \}_{i=1}^n $ of elements of $C(X)$ exist such that $ g = \sum_{i=1}^n h_i f_i $, then 
	\[ \Co(g) = \Co \left( \sum_{i=1}^n h_i f_i \right) \subseteq  \bigcup_{i=1}^n \Co(h_i f_i) \subseteq \bigcup_{i=1}^n \Co(f_i) \subseteq \O(S) \]
	Hence $ \O(I) = \bigcup_{g \in I} \Co(g) \subseteq \O(S) $ and consequently $ \O(I) = \O(S) $. 
\end{proof}

\begin{Pro}
	Let $\{ I_\alpha \}_{\alpha \in A} $ be a family of ideals of $C(X)$, $I$ and $J$ be ideals of $C(X)$, $\{ U_\alpha \}_{\alpha \in A} $ be a family of subsets of $X$ and $U$ and $V$ be subsets of $X$. Then the following hold.
	\begin{itemize}
		\item[(a)] $ \O\left(\sum_{\alpha \in A} I_\alpha \right) = \bigcup_{\alpha \in A} \O(I_\alpha) $.
		\item[(b)] $ \O \left( \bigcap_{\alpha \in A} I_\alpha \right) \subseteq  \bigcap_{\alpha \in A} \O (I_\alpha) $.
		\item[(c)] $ \I \left( \bigcup_{\alpha \in A} U_\alpha \right) = \bigcap_{\alpha \in A} \I(U_\alpha) $.
		\item[(d)] $ \O(I\cap J) = \O(I) \cap \O(J) $.
		\item[(e)] $ \I(U \cap V) \supseteq \I(U) + \I(V) $.
	\end{itemize}
\label{intersection and union in O and I}
\end{Pro}
\begin{proof}
	
	(a). By Proposition \ref{G generated}, 
	\[ \O \left(\sum_{\alpha \in A} I_\alpha \right) = \O \left( \bigcup_{\alpha \in A} I_\alpha \right) = \bigcup_{f \in \bigcup_{\alpha \in A} I_\alpha} \Co(f) = \bigcup_{\alpha \in A} \bigcup_{f \in I_\alpha} \Co(f) = \bigcup_{\alpha \in A} \O(I_\alpha).  \]
	
	(b). It follows immediately from Proposition \ref{order}.
	
	(c). Lemma \ref{order}, concludes that $ \I\left( \bigcup_{\alpha \in A} U_\alpha \right) \subseteq \bigcap_{\alpha \in A} \I(U_\alpha) $. 
	\begin{align*}
		f \in \bigcap_{\alpha \in A} \I(U_\alpha) \quad & \Rightarrow \quad f \in \I(U_\alpha) \quad \forall \alpha \in A \\
												  		& \Rightarrow \quad U_\alpha \subseteq \Ze(f) \quad \forall \alpha \in A \\
												  		& \Rightarrow \quad \bigcup_{\alpha \in A} U_\alpha \subseteq \Ze(f) \\
												  		& \Rightarrow \quad f \in \I \left(\bigcup_{\alpha \in A} U_\alpha \right)
	\end{align*}
	Thus $ \bigcap_{\alpha \in A} \I(U_\alpha) \subseteq \I\left( \bigcup_{\alpha \in A} U_\alpha \right) $ and consequently $ \I\left( \bigcup_{\alpha \in A} U_\alpha \right) = \bigcap_{\alpha \in A} \I(U_\alpha) $.
	
	(d). By Lemma \ref{order}, $ \O(I \cap J) \subseteq \O(I) \cap \O(J) $. 
	\begin{align*}
		\O(I) \cap \O(J) & = \left( \bigcup_{f\in I} \Co(f) \right) \cap \left( \bigcup_{g\in I} \Co(g) \right) \\
						 & = \bigcup_{\substack{f \in I \\ g \in J}} \left(\Co(f) \cap \Co(g)\right) = \bigcup_{\substack{f \in I \\ g \in J}} \Co(fg) \\
						 & \subseteq \bigcup_{ h \in I \cap J } \Co(h)  = \O(I\cap J)
	\end{align*}   
	Thus the equality holds.
	
	(e). It is clear, by Lemma \ref{order}.
\end{proof}

In the following examples we show that the equality in parts (b) and (e) of the above proposition need not establish. 

\begin{Exa}
	Consider $C(\R)$. For each $r \in \Q$, we have $ \O(M_r) = \R \setminus \{r\} $, thus $ \bigcap_{r \in Q}  \O(M_r) = \R \setminus \Q $. Also $ \bigcap_{r \in \Q} M_r = M_\Q = \{ 0 \} $, so $ \O \left( \bigcap_{r \in \Q} M_r \right) = \O ( \{ 0 \} ) = \R $.   
\end{Exa}

\begin{Exa}
	Consider $ C(\R) $. Easily we can see that, $ \I(\Q) = \{ 0 \} = \I(\R \setminus \Q) $, and thus $ \I( \Q \cap \R \setminus \Q ) = \I ( \emptyset ) = C(\R) \neq \{ 0 \} = \I(\Q) + \I(\R \setminus \Q) $.
\end{Exa}

\begin{Cor}
	Let $U$ and $V$ be subsets of $X$. The following are equivalent.
	\begin{itemize}
		\item[(a)] $ U \cup V $ is dense in $X$.
		\item[(b)] $ \I(U) \cap \I(V) = \{ 0 \} $.
		\item[(c)] $ \I(U) \I(V) = \{ 0 \} $.
	\end{itemize}
\label{I(U) I(V)=0}
\end{Cor}
\begin{proof}
	It follows from Lemma \ref{order} and Proposition \ref{intersection and union in O and I}.
\end{proof}

\begin{Pro}
	Let $ I $ be an ideal of $C(X)$ and $U$ be a subset of $X$. The following hold.
	\begin{itemize}
		\item[(a)] $ \O(\I(U)) = ( X \setminus U )^\circ $.
		\item[(b)] $\I(\O(I))=\An(I)$.
		\item[(c)] $ ( \I \O )^3 (I) = ( \I \O ) (I) $.
		\item[(d)] $ \O(\An(I)) = ( X \setminus \O(I) )^\circ $. 
	\end{itemize}
\label{O and I}
\end{Pro}
\begin{proof}
	(a). $ \displaystyle \O(\I(U)) = \bigcup_{ f \in \I(U) } \Co(f) = \bigcup_{ \Ze(f) \supseteq U} \Co(f) = \bigcup_{ \Co(f) \subseteq X \setminus U } \Co(f) = ( X \setminus U )^\circ  $.
	
	(b). \vspace{-0.5cm}	\begin{align*}
	f \in \An(I) \quad & \Leftrightarrow \quad \forall g \in I \quad fg= 0 \\
	& \Leftrightarrow \quad \forall g \in I \quad \Ze(f) \cup \Ze(g) = X \\
	& \Leftrightarrow \quad \forall g \in I \quad \Co(g) \subseteq \Ze(f) \\
	& \Leftrightarrow \quad \bigcup_{g \in I} \Co(g) \subseteq \Ze(f) \\
	& \Leftrightarrow \quad \O(I) \subseteq \Ze(f) \\
	& \Leftrightarrow \quad f \in \I(\O(I))  
	\end{align*}
	Thus $ \An(I) = \I(\O(I)) $.

	(c). Since $ \An^3(I) = \An(I) $, it follows from (c), immediately.
	
	(d). By (b) and (a), $ \O(\An(I)) = \O(\I(\O(I))) = ( X \setminus \O(I) )^\circ $.  
\end{proof}

\begin{Cor}
	If $I$ is a zero divisor ideal of $C(X)$, then $I$ is a fixed ideal.
\label{zero diviosr is fixed}
\end{Cor}
\begin{proof}
	Since $I$ is a zero divisor, $ \An(I) \neq \{ 0 \} $, so Proposition \ref{O and I}, deduces $ \I(\O(I)) \neq \{ 0 \} $, hence $\O(I)$ is not dense, by Lemma \ref{order}. Thus $ \O(I) \neq X $ and therefore $ \bigcap_{f \in I} \Ze(f) = X \setminus \O(I) \neq \emptyset $, so $I$ is a fixed ideal.
\end{proof}

The converse of the above corollary need not be true because for instance $ M_0 \subseteq C(\R) $ is a fixed ideal which is not a zero divisor ideal.

\begin{Cor}
	Let $P$ be a prime ideal of $C(X)$. $P$ is a zero divisor \ff there is some isolated point $p$ in $X$ such that $P = M_p = O_p$.
\end{Cor}
\begin{proof}
	($\Rightarrow$). By Corollary \ref{zero diviosr is fixed}, $P$ is fixed,  so there is some $p \in X$ such that $O_p \subseteq P \subseteq M_p $. Thus $ \O(P) = X \setminus \left( \bigcap_{f \in P} \Ze(f) \right) = X \setminus \{ p \} $. Since $P$ is a zero divisor, $\An(P) \neq \{ 0 \} $ and therefore $\I(\O(I)) \neq \{ 0 \} $, by Proposition \ref{O and I}. Now Lemma \ref{order}, deduces $ X \setminus \{ p \} $ is not dense and thus $ p $ is an isolated point. Consequently, $P = M_p = O_p $.
	
	($\Leftarrow$). Since $P = M_p $, $ \O(P) = X \setminus \left( \bigcap_{f \in P} \Ze(f) \right) = X \setminus \{p\} $. Since $p$ is an isolated point, it follows that $ \O(P) $ is not dense in $X$, thus $\I(\O(P)) \neq \{ 0 \}$, by Lemma \ref{order}. Now Proposition \ref{O and I}, follows that $ \An(P) \neq \{ 0 \}$ and therefore $P$ is a zero divisor.     
\end{proof}

\begin{Lem}
	If $G$ is an open subset of $X$, then an ideal $I$ exists such that $\O(I) = G$. i.e., $ \O $ maps the family of all ideals of $ C(X) $ onto the family of all open subsets of $ X $. 
\label{open in image}
\end{Lem}
\begin{proof}
	Put $ I = \Big< \big\{ f \in C(X) : \Co(f) \subseteq G \big\} \Big> $. Then by Proposition \ref{G generated},
	\[ \O(I) = \O(\Ge{ \{ f : \Co(f) \subseteq G \} } = \bigcup_{ \Co(f) \subseteq G } \Co(f) = G  \qedhere \]
\end{proof}

\begin{Lem}
	For each ideal $I$ of $C(X)$, we have $\O(I_z) = \O(I)$.
\label{G(I) = G(Iz)}
\end{Lem}
\begin{proof}
	Since $\Ze(I) = \Ze(I_z)$, so $ \{ \Co(f) : f \in I \} = \{ \Co(f) : f \in I_z \} $ and therefore $ \O(I_z) = \O(I) $.
\end{proof}

\begin{Thm}
	$ \O $ is a map from the family of all $z$-ideals of $C(X)$ onto the family of all open sets of $X$.
\end{Thm}
\begin{proof}
	It is clear by Lemmas \ref{open in image} and \ref{G(I) = G(Iz)}.
\end{proof}

\begin{Thm}
	Let $I$ and $J$ be two ideals of $C(X)$. The following statements hold
	\begin{itemize}
		\item[(a)] $ I J = \{ 0 \} $ \ff $ \O(I) \cap  \O(J) = \emptyset $.
		\item[(b)] $ I \An(J) = \{ 0 \} $ \ff $ \O(I) \subseteq \overline{\O(J)} $.
		\item[(c)] $ \An(I) \An(J) = \{ 0 \} $ \ff $ \overline{\O(I) \cup \O(J)} = \emptyset $.
		\item[(d)] $ \overline{\O(I)} = \overline{\O(J)} $ \ff $ \An(I) = \An(J) $.
		\item[(e)] $  \I(U) I = \{ 0 \} $ \ff $ \O(I) \subseteq \overline{U} $.
	\end{itemize}
\label{IJ=0 and G}
\end{Thm}
\begin{proof}
	(a $\Rightarrow$). Since $IJ = \{ 0 \} $, $ I \subseteq \An(J) $, thus $ I \subseteq \I(\O(J)) $, by Proposition \ref{O and I}. Now suppose that $ f \in I $, then $ f \in \I(\O(J)) $, hence $ \Ze(f) \supseteq \O(J) $, so $\Co(f) \subseteq X \setminus \O(I) $. This follows that $ \O(I) = \bigcup_{f\in I} \Co(f) \subseteq X \setminus \O(I) $ and therefore $\O(I) \cap \O(J) = \emptyset$.
	
	(a $\Leftarrow$). Suppose that $ f \in I $ and $ g \in J$, then $ \Co(f) \subseteq \O(I) $ and $ \Co(g) \subseteq \O(J)$, thus $ \Co(f) \cap \Co(g) \subseteq \O(I) \cap \O(J) = \emptyset $, so $ f g = 0 $ and therefore $ I J = \{ 0 \} $.
	
	(b). By part (a), $ I \An(J) = \{0\} $ \ff $ \O(I) \cap \O(\An(J)) = \emptyset $. By Proposition \ref{O and I}, it is equivalent to $ \O(I) \cap \left( X \setminus \O(I) \right)^\circ = \emptyset $. One can see easily that, it is equivalent to say that $ \O(I) \subseteq \overline{\O(I)} $.
	
	(c). By part (a), $ \An(I) \An(J) = \{0\} $ \ff $ \O(\An(I)) \cap \O(\An(J)) = \emptyset $; \ff, $ \left( X \setminus \O(J) \right)^\circ \cap \left( X \setminus \O(I) \right)^\circ = \emptyset $, By Proposition \ref{O and I}. It easy to see that, it is equivalent to say that $ \overline{\O(I) \cup \O(I)} = X $.
	
	(d ). By part (b), 
	\begin{align*}
	\overline{\O(I)} = \overline{\O(J)} \quad & \Leftrightarrow \quad \O(I) \subseteq \overline{\O(J)} \; \text{ and } \; \O(J) \subseteq \overline{\O(I)} \\
								& \Leftrightarrow \quad I \An(J) = \{ 0 \} \; \text{ and } \; J \An(I) = \{ 0 \} \\
								& \Leftrightarrow \quad \An(J) \subseteq \An(I) \; \text{ and } \; \An(I) \subseteq \An(J) \\
								& \Leftrightarrow \quad \An(I) = \An(J) \; . \qedhere
	\end{align*}
	
	(e). By part (a) and Proposition \ref{O and I},
	\begin{align*}
	I \I(U) = \{0\} \quad & \Leftrightarrow \quad \O(I) \cap \O(\I(U)) = \emptyset\\
						  & \Leftrightarrow \quad \O(I) \cap ( X \setminus U)^\circ = \emptyset \\
						  & \Leftrightarrow \quad \O(I) \cap X \setminus \overline{U} = \emptyset  \\
						  & \Leftrightarrow \quad \O(I) \subseteq \overline{U} \; . \qedhere
	\end{align*}
\end{proof}

By the above theorem two ideals $ I $ and $ J $ are adjacent \ff each maximal ideal of $ C(X) $ contains either $ I $ or $ J $. Now we can conclude the following corollary from Lemmas \ref{order} and \ref{open in image} and Theorem \ref{IJ=0 and G}.

\begin{Cor}
	Suppose that $I$  is a non-zero ideal of $C(X)$ and $ U \subseteq X $.
	\begin{itemize}
		\item[(a)] $I \in \AC$ \ff $ \overline{\O(I)} \neq X$.
		\item[(b)] $ \I(U) \in \AC $ \ff $ \overline{U}^{^\circ} \neq \emptyset $.
	\end{itemize}
	\label{element of AG}
\end{Cor}

\begin{Cor}
	Suppose that $ I,J \in \AC $. $ I $ and $ J $ are orthogonal \ff $ \O(I) \cap \O(J) = \emptyset $ and $ \overline{ \O(I) \cup \O(J)} = X $.
\label{orthogonal}
\end{Cor}
\begin{proof}
	It is easy to verify, by Theorem \ref{IJ=0 and G} and Corollary \ref{element of AG}.
\end{proof}

\begin{Pro}
	If the image of the mapping $\O$ is the family of all cozero sets of $X$, then $C(X)$ is an i.a.c ring.
\end{Pro}
\begin{proof}
	Suppose that $S$ is a subset of $X$ and set $I = \big< S \big>$. For some $f \in C(X)$, we have $\O(I) = \O(S) = \Co(f) $, by the assumption and Proposition \ref{G generated}. Thus by Lemma \ref{order}, Proposition \ref{G generated} and Theorem \ref{IJ=0 and G}, 
	\begin{align*}
	g \in \An(I) \quad & \Leftrightarrow \quad g I = \{ 0 \} \quad \Leftrightarrow \quad \O ( \big< g\big> ) \cap \O(I) = \emptyset \\
					   & \Leftrightarrow \quad \Co(g) \cap \Co(f) = \emptyset \quad \Leftrightarrow \quad gf = 0 \quad \Leftrightarrow \quad  g \in \An(f) 
 	\end{align*}
 	Hence  $ \An(S) = \An(I) = \An(f) $, i.e. $C(X)$ is an i.a.c. ring.
\end{proof}

\section{Radius of the graph}

In this section, some topological properties of $ X $ are linked to distance and eccentricity of vertices of $ \AGC $, then by these facts we study the radius of the graph. 

\begin{Lem}
	For any ideals $I$ and $J$ in $\AC$,
	\begin{itemize}
		\item[(a)] $ d(I,J) = 1 $ \ff $ \O(I) \cap \O(J) = \emptyset $.
		\item[(b)] $ d(I,J) = 2 $ \ff $ \O(I) \cap \O(J) \neq \emptyset $ and $ \overline{\O(I) \cup \O(J)} \neq X $.
		\item[(c)] $ d(I,J) = 3 $ \ff $ \O(I) \cap \O(J) \neq \emptyset $ and $ \overline{\O(I) \cup \O(J)} = X $.
	\end{itemize}
\label{distance}
\end{Lem}
\begin{proof}
	(a). It is clear, by Theorem \ref{IJ=0 and G}.
	
	(b $\Leftarrow$). Since $I$ is not adjacent to $J$, $ \O(I) \cap \O(J) \neq \emptyset $, by Theorem \ref{IJ=0 and G}. By the assumption there is an ideal $ K $ in $ \AC $ such that $ K $ is adjacent to both ideals $ I $ and $ J $, so Lemma \ref{order}, concludes that $ \O(K) \neq \emptyset $ and also Theorem \ref{IJ=0 and G}, deduces that $ \O(I) \cap \O(K) = \emptyset $ and $ \O(J) \cap \O(K) = \emptyset $, thus $ \O(K) \cap \left( \O(I) \cup \O(J) \right) = \emptyset $, hence $ \overline{\O(I) \cup \O(J)} \neq X $.
	
	(b $\Leftarrow$). Theorem \ref{IJ=0 and G}, follows that $I$ is not adjacent to $J$. Set $ H = \O(I) \cup \O(J) $ and $ K = \I(H) $. Since $ \emptyset \neq H \subseteq \overline{H}^{^\circ} \neq \emptyset $, by Lemma \ref{open in image} and Corollary \ref{element of AG}, $ \I(H) \in \AC $. Since $ \O(I) , \O(J) \subseteq H \subseteq \overline{H} $, $ I K = J K = \{0\} $, by Theorem \ref{IJ=0 and G}. Hence $ K $ is adjacent to both ideals $ I $ and $ J $, thus $ d(I,J) =3 $.
	
	(c). It follows from (a), (b) and \cite[Theorem 2.1]{behboodi2011annihilating}. 
\end{proof}

\begin{Lem}
	Let $f \in C(X)$ and $I$ be an ideal of $C(X)$ and $p \in \O(I)$. If $ \Co(f) \subseteq \{ p \} $ and $p$ is an isolated point of $X$, then $ f \in I $.
\end{Lem}
\begin{proof}
	Since $ p \in \O(I) $, there is some $g \in I$ such that $p \in \Co(g)$. Set
	\[ h(x)  =  \begin{cases}
					\frac{f(p)}{g(p)} & x = p \\
					0 & x \neq p 
				\end{cases} \]
	Since $p$ is an isolated point, $h \in C(X)$. Now we have $ f = g h $ and therefore $ f \in I $.
\end{proof}

\begin{Pro}
	Suppose that $I$ is a non-zero annihilating ideal of $C(X)$. The following statements hold.
	\begin{itemize}
		\item[(a)] $\Ec(I) = 3$ \ff $\O(I)$ is not singleton. 
		\item[(b)] $\Ec(I) = 2$ \ff $\O(I)$ is singleton and $|X| > 2$.
		\item[(c)] $\Ec(I) = 1$ \ff $\O(I)$ is singleton and $|X| = 2$.
	\end{itemize}
\label{ecc}
\end{Pro}
\begin{proof}
	(a $\Rightarrow$). There is some $J \in \AC$ such that $d(I,J) = 3$. Lemma \ref{distance}, concludes that $\O(I) \cap \O(J) \neq \emptyset $ and $\overline{\O(I) \cup \O(J)} = X$. If $ \O(I)$ is singleton, then $ \O(I) \subseteq \O(J) $ and therefore $ \overline{\O(J)} = \overline{\O(I) \cup \O(J)} = X $, so $J \notin \AC $, by Corollary \ref{element of AG}, which is a contradiction.

	(a $\Leftarrow$). There are distinct points $p$ and $q$ in $\O(I)$, so there are disjoint open sets $H,K \subseteq \O(I)$ such that $p \in H$ and $q \in K$. By Lemma \ref{open in image}, there is some ideal $J$ such that $\O(J) = H \cup X \setminus \overline{\O(I)} $. Since $ q \notin \overline{\O(J)} $ and $ p \in \O(J) $, Lemma \ref{order} and Corollary \ref{element of AG}, concludes that $J \in \AC$. Then
	\[ H \subseteq \O(I) \cap \O(J) \quad \Rightarrow \quad \O(I) \cap \O(J) \neq \emptyset \]
	\[ \overline{\O(I) \cup \O(J)} \supseteq \overline{\O(I)} \cup \left( X \setminus \overline{\O(I)} \right) = X \quad \Rightarrow \quad \overline{\O(I) \cup \O(J)} = X \]
	Hence $d(I,J) = 3$, by Lemma \ref{distance}. Consequently, $\Ec(I) = 3$.

	(c $\Rightarrow$). There is some ideal $I \in \AC$ which is adjacent to any element of $\AC$. By (a), $\O(I)$ is singleton, thus there is some isolated point $p \in X$ such that $\O(I) =\{p\}$. Since $ \emptyset \neq X \setminus \{p\} $ is open, by Lemma \ref{open in image}, there is some ideal $J$, such that $\O(J) = X \setminus \{ p \}$. Since $ \O(J) \neq \emptyset $ and $\overline{\O(J)} =  X \setminus \{p\} \neq \emptyset$, $ J \in \AC$, by Lemma \ref{order} and Corollary \ref{element of AG}. Since $ \Ec(I) = 1 $, $\Ec(J) \leqslant 2$, so $\O(J)$ is singleton, by part (a), and therefore $|X|=2$.

	(c $\Leftarrow$). $C(X) \cong \R \oplus \R$, so $\AGC$ is a star graph, by \cite[Corollary 2.3]{behboodi2011annihilating}, since $ \AC $ has just two element, it follows that $ \Ec(I) = 1 $.
	
	(b). It concludes from (a) and (c).
\end{proof}

The following corollary is an immediate consequences of the above theorem. 

\begin{Cor}
	$|X| = 2$ \ff $\AGC$ is star.
\label{star}
\end{Cor}

Now we can determine the radius of the graph.

\begin{Thm}
	For any topological space $ X $, 
		\[  \Ra(\AGC) = \begin{cases}
		1 & \text{ if } |X| = 2 \\
		2 & \text{ if } |X| > 2 \text{ and } X \text{ has an islated point.} \\
		3 & \text{ if } |X| > 2 \text{ and } X \text{ has no any isloated point.}
		\end{cases} \]
\label{radius}
\end{Thm}
\begin{proof}
	It is straight consequence of Lemma \ref{open in image} and Proposition \ref{ecc}.
\end{proof}

\section{Girth of the graph}

In this section, first we correspond an equivalent topological property to leaf vertices, then we show that if $ \AGC $ has a cycle then $ \Gi\AGC=3 $. Finally we try to associate the graph properties of $\AGC$, the ring properties of $C(X)$ and the topological properties of $X$.

\begin{Lem}
	Suppose that $ Y $ is a clopen subset of $ X $. For each ideal $ I $ of $ C(X) $, there are ideals $ I_1 $ and $ I_2 $ of $ C(X) $ such that $ I = I_1 \oplus I_2 $ and $ I_1 $ and $ I_2 $ are ideals of $ M_Y \cong C(X \setminus Y) $ and $ M_{ X \setminus Y} \cong C(Y) $, respectively.
\label{sum of ideals}
\end{Lem}
\begin{proof}
	By this fact that since $ Y $ is clopen, $ C(X) \cong C(Y) \oplus C(X \setminus Y) $, it is straightforward.
\end{proof}

\begin{Pro}
	Let $ I \in \AC $. $ X \setminus \overline{\O(I)} $ is singleton \ff $ I $ is a leaf vertex.
\label{leaf vertex}
\end{Pro}
\begin{proof}
	$ \Rightarrow $). Suppose that $ X \setminus \overline{\O(I)} = \{p\} $. Since $ \{ p \}  $ is open, by Lemma \ref{open in image}, there is an ideal $ J $ such that $ \O(J) = \{ p \} $, then $ \overline{\O(J)} = \{p\} $, and therefore $ J \in \AC $, by Lemma \ref{order} and Corollary \ref{element of AG}. Also $ \O(I) \cap \O(J) = \emptyset $, so $ I $ is adjacent to $ J $, by Theorem \ref{IJ=0 and G}. Suppose that $ K $ is adjacent to $ I $ and $ Y  = \overline{\O(I)} $. Then $ \O(K) \cap \O(I) = \emptyset $, by Theorem \ref{IJ=0 and G}, thus $ \O(K) \subseteq X \setminus \overline{\O(I)} = \{ p \} $. By Lemma \ref{order}, $ \O(K) \neq \emptyset $, so $ \O(K) = \{ p \} $. Since $ \{ p \} $ is clopen, by Lemma \ref{sum of ideals}, it follows that there are ideals $ K_1 $ and $ K_2 $ of $ M_p \cong C(Y) $ and $ M_Y \cong C\left( \{ p \} \right) \cong \R $, respectively, such that $ K = K_1 \oplus K_2 $. If $ K_1 \neq \{ 0 \} $, then $ 0 \neq f \in K_1 \subseteq K $ exists, so there is some $ q \in Y $ such that $ f(q) \neq 0 $, thus $ p \neq q \in \O(K) $, which is a contradiction. Hence $ K_1 = \{ 0 \} $, since $ K \neq \{ 0 \} $, it follows that $ K_2 = M_Y $, thus $ K = M_Y $, and this completes the proof.
	
	$ \Leftarrow $). Suppose that $ X \setminus \overline{\O(I)} $ is not singleton, so distinct points $ p , q $ in $ X \setminus \overline{\O(I)} $ exist. Using Lemma \ref{order} and Corollary \ref{element of AG}, it is easily to verify that there are disjoint open sets $ H_1 $ and $ H_2 $ containing $ p $ and $ q $, respectively, which $ H_1 \cap \O(I) = H_2 \cap \O(I) = \emptyset $. Now Lemma \ref{open in image}, implies that there are ideals $ J_1 $ and $ J_2 $ such that $ \O(J_1) = H_1 $ and $ \O(J_2) = H_2 $, clearly $ J_1,J_2 \in \AC $. Then $ \O(I) \cap \O(J_1) = \O(I) \cap \O(J_2) = \emptyset $. So, by Theorem \ref{IJ=0 and G}, $ I $ is adjacent to both ideals $ J_1 $ and $ J_2 $.
\end{proof}

\begin{Lem}
	Suppose that $ I , J \in \AC $ and they are not leaf vertices. The following statements hold.
	\begin{itemize}
		\item[(a)] $ \O(I) \cap \O(J) = \emptyset $ and $ \overline{\O(I) \cup \O(J)} \neq X $ \ff $ \gi(I,J) = 3 $.
		\item[(b)] If $ \O(I) \cap \O(J) = \emptyset $ and $ \overline{\O(I) \cup \O(J)} = X $, then $ \gi(I,J) = 4 $.
		\item[(c)] If $ \O(I) \cap \O(J) \neq \emptyset $ and $ \overline{\O(I)} = \overline{\O(J)} $, then $ \gi(I,J) = 4 $.
		\item[(d)] Suppose that $ \O(I) \cap \O(J) \neq \emptyset $ and $ \overline{\O(I)} \neq \overline{\O(J)} $. Then $ X \setminus \overline{\O(I) \cup \O(J)} $ is not singleton \ff $\gi(I,J) = 4 $.
		\item[(e)] $ \O(I) \cap \O(J) \neq \emptyset $, $ \overline{\O(I)} \neq \overline{\O(J)} $ and $ X \setminus \overline{\O(I) \cup \O(J)} $ is singleton \ff $\gi(I,J) = 5 $.
	\end{itemize}
\label{gi}
\end{Lem}
\begin{proof}
	(a $\Rightarrow$). Set $ H = \O(I) \cup \O(J) $ and $ K = \I(H) $. Since $ \overline{H} \neq X  $ and $ \overline{H}^{^\circ} \neq \emptyset $, $ K \in \AC $, by Lemma \ref{order} and Corollary \ref{element of AG}. Since $ \O(I) , \O(J) \subseteq H \subseteq \overline{H} $, by Theorem \ref{IJ=0 and G}, $ K $ is adjacent to both ideals $ I $ and $ J $. By the assumption and Theorem \ref{IJ=0 and G}, $ I $ adjacent to $ J $, hence $ \gi(I,J) = 3 $.   

	(a $\Leftarrow$). By the assumption, $I$ is adjacent to $J$ and some $ K \in \AC $ exists such that $K$ is adjacent to both ideals $I$ and $J$, so $\O(I) \cap \O(J) = \emptyset $, $ \O(I) \cap \O(K) = \emptyset $ and $ \O(J) \cap \O(K) = \emptyset $, by Theorem \ref{IJ=0 and G}. Hence $ \left( \O(I) \cup \O(J) \right) \cap \O(K) = \emptyset $. Since $K \neq \{0\} $, $ \O(K) \neq \emptyset$, by Lemma \ref{order}, and therefore $ \overline{\O(I) \cup \O(J)} \neq X $.
	
	(b). The assumption and  part (a) imply that $ \gi(I,J) \geqslant 4  $ and Theorem \ref{IJ=0 and G}, concludes that $I$ is adjacent to $ J $ and $ \An(I) $ is adjacent to $ \An(J) $. Since $ I $ and $ J $ is adjacent to $ \An(I) $ and $ \An(J) $, respectively, the proof is complete.
	
	(c). We can conclude from the assumption and part (a), that $ \gi(I,J) \geqslant 4  $. Since $ \O(I) \subseteq \overline{\O(J)} $ and $ \O(J) \subseteq \overline{\O(J)} $, by Theorem \ref{IJ=0 and G}, $ I $ and $ J $ are adjacent to $ \An(J) $ and $ \An(I)  $, respectively. Since $ I $ and $ J $ are adjacent to $ \An(I) $ and $ \An(J) $, respectively, the proof is complete.
	
	(d $ \Rightarrow $). It easy to see that there are two distinct open sets $ H_1 $ and $ H_2 $ such that $ H_1 \cap \O(I) = H_1 \cap \O(J) = H_2 \cap \O(I) = H_2 \cap \O(J) =\emptyset  $. Then, by Lemma \ref{open in image}, there are two ideals $ K_1 $ and $ K_2 $ such that $ \O(K_1) = H_1 $ and $ \O(K_2) = H_2 $, it is clear that $K_1,K_2 \in \AC$, by Lemma \ref{order} and Corollary \ref{element of AG}. Now Theorem \ref{IJ=0 and G}, concludes that both vertices $ I $ and $ J $ are adjacent to both vertices $ K_1 $ and $ K_2 $, thus $ \gi(I,J) = 4 $, by part (a).
	
	(d $ \Leftarrow $). By Theorem \ref{IJ=0 and G}, $ I $ is adjacent to $ J $, since $ \gi(I,J) = 4 $, it follows that there are distinct vertices $ K_1 $ and $ K_2 $ which are adjacent to both vertices $ I $ and $ J $, so $ I + J $ is adjacent to both vertices $ K_1 $ and $ K_2 $. Now Propositions \ref{intersection and union in O and I} and \ref{leaf vertex}, conclude that $ X \setminus \overline{\O(I) \cup \O(J)} = X \setminus \overline{\O(I+J)} $ is not singleton.
	
	(e $ \Rightarrow $). By parts (a) and (d), $ \gi(I,J) \geqslant 5 $. If $ \O(I) \subseteq \overline{\O(J)}  $, then $ \overline{\O(I)} \subseteq \overline{\O(J)} $, so $ X \setminus \overline{\O(I) \cup \O(J)} = X \setminus \overline{ \O(J)} $ and therefore $ X \setminus \overline{\O(J)}  $ is singleton, by the assumption. Now Proposition \ref{leaf vertex}, concludes that $ J $ is a leaf vertex, which contradicts the assumption, so $ \O(I) \not \subseteq \overline{\O(J)} $, one can show similarly that $ \O(J) \not \subseteq \overline{\O(I)} $, so $ H_1 = \O(I) \setminus \overline{\O(J)} $ and $ H_2 = \O(J) \setminus \overline{\O(I)} $ are nonempty open sets, thus, Lemma \ref{open in image}, implies that there are ideals $ K_1 $ and $ K_2 $ such that $ \O(K_1) = H_1 $ and $ \O(K_2) = H_2 $, it is evident that $ K_1,K_2 \in \AC $, by Lemma \ref{order} and Corollary \ref{element of AG}. Since $ X \setminus \overline{\O(I) \cup \O(J)} $ is nonempty open set, there is an ideal $ K_3 $ such that $ \O(K_3) = X \setminus \overline{\O(I) \cup \O(J)} $, it is clear that $ K_3 \in \AC $, by Lemma \ref{order} and Corollary \ref{element of AG}. Then 
	\begin{align*}
	\O(I) \cap \O(K_2) & = \O(K_2) \cap \O(K_1) = \O(K_1) \cap \O(J) \\
					   & = \O(J) \cap \O(K_3) = \O(K_3) \cap \O(I) = \emptyset 
	\end{align*}
	so $ \gi(I,J) = 5 $.
	
	(e $ \Leftarrow $). It is clear, by Parts (a)-(d).
\end{proof}

It is clear that if $ |X| = 2 $, then $ \AGC $ has no any cycle. In the following theorem we show that if $ \AGC $ has an cycle then the girth of the graph is 3.

\begin{Thm}
	If $ |X| > 2 $, then $ \Gi \AGC = 3 $.
\label{girth}
\end{Thm}
\begin{proof}
	It easy to see that there are mutually disjoint nonempty open sets $ G_1 $, $ G_2 $ and $ G_3 $. By Lemma \ref{open in image}, there are ideals $ I_1 $, $ I_2 $ and $ I_3 $, such that $ \O(I_1) = G_1 $, $ \O(I_2) = G_2 $ and $ \O(I_3) = G_3 $, evidently, $ I_1,I_2,I_3 \in \AC $, by Lemma \ref{order} and Corollary \ref{element of AG}. By Theorem \ref{IJ=0 and G}, $ I_1 $ is adjacent to $ I_2 $, $ I_2 $ is adjacent to $ I_3 $ and $ I_3 $ is adjacent to $ I_1 $, hence $ \Gi \AGC =3 $.
\end{proof}

\begin{Thm}
	The following statements are equivalent.
	\begin{itemize}
		\item[(a)] $ X $ has an isolated point.
		\item[(b)] $ \R $ is a direct summand of $ C(X) $.
		\item[(c)] $ \AGC $ has a leaf vertex.
		\item[(d)] $ \AGC $ is not triangulated.
	\end{itemize}
\end{Thm}
\begin{proof}
	(a $\Leftrightarrow$ b) and (c $ \Rightarrow $ d) are clear and (a $ \Leftrightarrow $ c) follows from Proposition \ref{leaf vertex}.
	
	(d $\Rightarrow$ a) Suppose that $ X $ has no any isolated point and $ I \in \AC $. Then $ X \setminus \overline{\O(I)} $ is not singleton, so it has two distinct points $ p $ and $ q $, so there are disjoint open sets $ G_1 $ and $ G_2 $, such that $ G_1 \cap \O(I) = G_2 \cap \O(I) = \emptyset $. By Lemma \ref{open in image}, there are $ J,K \in \AC $, such that $ \O(J) = G_1 $ and $ \O(K) = G_2 $. Thus $ I $ is adjacent to $ J $, $ J $ is adjacent to $ K $ and $ K $ is adjacent $ I $. Consequently, $ \AGC $ is triangulated.
\end{proof}

\section{Dominating number}

In the last section, we give an upper bound and a lower bound for dominating number of the graph by topological notions, then the chromatic number and the clique number of the graph are studied. Finally, we introduce the disjoint open set graph on a topology and give the radius, dominating number, diameter and the girth of this graph.

\begin{Thm}
	$ c(X) \leqslant \dt(\AGC) \leqslant w(X) $, for each topological space $ X $.
\label{cellularity < dominating}
\end{Thm}
\begin{proof}
	Suppose that $ \mathcal{U} $ is a family of mutually disjoint nonempty open sets. If $ \overline{\bigcup \mathcal{U}} \neq X $, then $ \mathcal{V} = \mathcal{U} \cup \left\{ X \setminus \overline{\bigcup \mathcal{U}} \right\} $ is a family of mutually open sets  which $ \overline{\bigcup \mathcal{V}} = X $, so without loss of generality we can assume that $ \overline{\bigcup \mathcal{U}} = X $. For each $ U \in \mathcal{U} $, there is some $ I_{_U} \in \AC $ such that $ \O(I_{_U}) = U $, by Lemma \ref{open in image}. Since $ U \neq \emptyset $ and $ \overline{U} \neq X $, it follows that $ I_{_U} \in \AC $, by Lemma \ref{order} and Corollary \ref{element of AG}. Now suppose that $ D $ is a dominating set, then for each $ U \in \mathcal{U} $, there is some ideal $ J_{_U} $ in $ D $ such that $ J_{_U} $ is adjacent to $ \sum_{U \neq V \in \mathcal{U}} I_{_V} $, now Theorem \ref{IJ=0 and G}, implies that $\O(J_{_U}) \cap \O\left(\O \left( \sum_{U \neq V \in \mathcal{U}} I_V \right)  \right) = \emptyset $, thus $ \O(J_{_U}) \cap \left( \bigcup_{U \neq V \in \mathcal{U}} U \right) = \emptyset $. Suppose that $ J_{_U} = J_{_{U'}} $, then $ \O(J_{_U}) = \O(J_{_{U'}}) $. If $ U \neq U' $, then 
	\begin{align*}
	\O(J_{_U}) \cap \bigcap \mathcal{U} & = \O(J_{_U}) \cap \bigg[ \Big( \bigcup_{U \neq V \in \mathcal{U}} V \Big) \cup \Big( \bigcup_{U' \neq V \in \mathcal{U}} V \Big) \bigg] \\
										& = \bigg[ \O(J_{_U}) \cap  \Big( \bigcup_{U \neq V \in \mathcal{U}} V \Big) \bigg] \cup \bigg[ \O(J_{_U}) \cap \Big( \bigcup_{U' \neq V \in \mathcal{U}} V \Big) \bigg] = \emptyset.
	\end{align*}  
	Thus $ \overline{\bigcup \mathcal{U}} \neq X $, which  contradicts our assumption. Hence $ U = U' $, so $ | \mathcal{U} | \leqslant | D | $, and consequently $ c(X) \leqslant \dt(\AGC) $.
	
	Now suppose that $ \mathcal{B} $ is a base for $ X $, without loss of generality we can assume that every element of $ \mathcal{B} $ is not empty. For each $ B \in \mathcal{B} $, there is some $ 0 \neq f_{_B} \in C(X) $ such that $ \emptyset \neq \Co(f_{B}) \subseteq B $. Clearly, we can choose $ f $ such that $ \overline{\Co(f_{_B})} \neq X $. Lemma \ref{order}, concludes that $ \O(\Ge{f_{_B}}) = \Co(f_{_B}) $, so $ \O(\Ge{f_{_B}}) \neq \emptyset $ and $ \overline{\O(\Ge{f_{_B}})} \neq X $, for each $ B \in \mathcal{B} $, thus $ \Ge{f_{_B}} \in \AC $, by Lemma \ref{order} and Corollary \ref{element of AG}. For each $ J \in \AC $, $ \overline{\O(I)} \neq X $, by Corollary \ref{element of AG}, so $ \left( X \setminus \O(I)\right)^\circ \neq \emptyset $, thus $ B \in \mathcal{B} $ exists such that $ B \subseteq \left(X \setminus \O(I)\right)^\circ $, hence $ \O(\Ge{f_{_B}}) \subseteq X \setminus \O(I) $, consequently , $ \O(\Ge{f_{_B}}) \cap \O(I) = \emptyset $, therefore Theorem \ref{IJ=0 and G}, implies that $ \Ge{f_{_B}} $ is adjacent to $ I $. Hence $ \{ \Ge{f_{_B}} : B \in \mathcal{B} \} $ is a dominating set, since $ \big| \{ \Ge{f_{_B}} : B \in \mathcal{B} \} \big| \leqslant | \mathcal{B} | $, it follows that $ \dt(\AGC) \leqslant w(X) $.
\end{proof}

\begin{Cor}
	If $ X $ is discrete, then $ \dt(\AGC) = |X| $.
\label{dominating and discrete}
\end{Cor}
\begin{proof}
	For each $ p \in X $, since $ \{ p \} $ is open, there is some $ I_p \in \AC $ such that $ \O(I_p) = \{ p \} $, by Lemma \ref{open in image}. For each $ I \in \AC $, $ \O(I) \neq X $, by Corollary \ref{element of AG}, thus $ p \in X \setminus \O(I) $ exists, hence $ \O(I) \cap \O(I_p) = \emptyset $, so $ I $ is adjacent to $ I_p $. This shows that $ \{ I_p : p \in X \} $ is a dominating set, this implies that $ \dt(\AGC) \leqslant | X | $. It is clear that $ | X | = c(X) $, so $ | X | \leqslant \dt(\AGC) $, by Theorem \ref{cellularity < dominating}. Consequently, $ \dt(\AGC) = |X| $.
\end{proof}

\begin{Thm}
	$ \dt(\AGC) $ is finite \ff $ | X | $ is finite. In this case, $ \dt(\AGC) = | X | $.
\end{Thm}
\begin{proof}
	$ \Rightarrow $). Suppose that $ | X  |$ is infinite. Clearly $ c(X) $ is infinite, so $ \dt(\AGC) $ is infinite, by  Theorem \ref{cellularity < dominating}.
	
	$ \Leftarrow $). If $ | X | $ is finite, then $ X $ is discrete, so $ \dt(\AGC) = |X| $ is finite, by Corollary \ref{dominating and discrete}.
\end{proof}

\begin{Thm}
	$ \chi \AGC = \Cli \AGC = c(X) $, for each topological space $ X $.
\label{cellularity and clique}
\end{Thm}
\begin{proof}
	It is an immediate consequence of Proposition \ref{diameter and clique=chi}, Lemma \ref{open in image} and Theorem \ref{IJ=0 and G}.
\end{proof}

\begin{Def}
	Suppose that $ (X,\tau) $ is a topological space, set $ \tau^* = \{ G \in \tau : G \neq \emptyset \text{ and } ( X \setminus G )^\circ \neq \emptyset \} $. We define a graph with vertices of elements of $ \tau^* $, where two distinct vertices $ G $ and $ H $ are adjacent \ff $ G \cap H = \emptyset $. We call this graph disjoint open set graph and denote by $ \DGX $.
\end{Def}

\begin{Lem}
	Suppose that $ G $ and $ G' $ are two graphs. If $ \varphi $ is a map from the vertices of $ G $ onto vertices $ G' $ such that $ \{u,v\} \in E(G) $ \ff $ \{\varphi(u),\varphi(v)\} \in E(G') $ and $ \{ u,v \} \in E(G) $ implies that $ \varphi(u) \neq \varphi(v) $, then
	\begin{itemize}
		\item[(a)] $ \Di(G') = \Di(G) $.
		\item[(b)] $ \Ra(G') = \Ra(G) $.
		\item[(c)] $ \Gi(G') \leqslant \Gi(G) $.
		\item[(d)] $ \dt(G') \leqslant \dt(G) $.
		\item[(e)] $ \Cli(G') = \Cli(G) $.
		\item[(f)] $ \chi(G') = \chi(G) $.
		\item[(g)] $ G $ is complemented \ff $ G' $ is complemented.
	\end{itemize}
\end{Lem}
\begin{proof}
	Suppose that $u,v \in V(G)$. $u = u_0 - u_1 - u_2 - \cdots - u_n = v $ is a path \ff $\varphi(u) = \varphi(u_0) - \varphi(u_1) - \cdots - \varphi(u_n) = \varphi(v) $ is  a path, so $ d(\varphi(u),\varphi(v)) = d(u,v) $. By this fact the proof of the statements (a),(b) and (c) are clear.
		
	(d). Readily, we can see that if $ A \subseteq V(G) $ is a dominating set, then $ \varphi(A) \subseteq V(G') $ is a dominating set, so $ \dt(G') \leqslant \dt(G) $. 
	
	(e). Similar to the proof part (d) we can show that $ \Cli(G') \leqslant \Cli(G) $. Now suppose that $ S \subseteq V(G') $ is a clique set. Put $ u_v \in \varphi^{-1}(u) $, clearly $ \{ u_v \}_{v \in S} $ is a clique set and $ \left| \{ u_v \}_{v \in S} \right| = | S | $, hence $ \Cli(G) \leqslant \Cli(G') $ and consequently the equality holds. 
	
	(f). Suppose that for some cardinal number $ x $, the map $ f : G' \rightarrow x $ is a coloring, (i.e., if $ \{u,v\} \in E(G)  $, then $ f(u) \neq f(v) $). Then, clearly, $ f \circ \varphi : G \rightarrow x $ is a coloring, so $ \chi(G') \leqslant \chi(G) $. Now suppose that $ g : G \rightarrow y $ is a coloring, for some cardinal number $ y $. For each $ v \in V(G) $, choose $ u_v \in V(G)  $ such that $ \varphi(u_v) = v $. Let $ \bar{g} : G' \rightarrow y  $ be given by $ \bar{g} (v) = g(u_v) $. Evidently, $ \bar{g} $ is a coloring, thus $ \chi(G') \leqslant \chi(G) $, and consequently $ \chi(G) = \chi(G') $.
	
	(g). It is easy.
\end{proof}

The following Theorem is an immediate consequence of the above Lemma, Proposition \ref{diameter and clique=chi}, Theorems \ref{radius}, \ref{girth}, and \ref{cellularity and clique} and this fact that the annihilating-ideal graph of a ring is complemented.

\begin{Thm}
	Let $ X $ be a topological space. The following statements hold.
	\begin{itemize}
		\item[(a)] \[  \Di(\DGX) =  \begin{cases}
								1 & \text{ if } |X| = 2 \\
								3 & \text{ if } |X| > 2 
									\end{cases} \]
		\item[(b)] $ |X| = 2 $ \ff $ \DGX $  is a star graph.
		\item[(c)] \[  \Ra(\DGX) = \begin{cases}
							1 & \text{ if } |X| = 2 \\
							2 & \text{ if } |X| > 2 \text{ and } X \text{ has an islated point.} \\
							3 & \text{ if } |X| > 2 \text{ and } X \text{ has no any isloated point.}
								\end{cases} \]
		\item[(d)] If $ |X| > 2 $, then $ \Gi(\DGX) = 3 $.
		\item[(e)] $ \chi (\DGX) = \Cli \DGX = c(X) $.
		\item[(g)] $ \DGX $ is complemented.
	\end{itemize}
\end{Thm}

Finally, we note that it is clear that $ \chi(\DGX) = c(X) $.

\bibliographystyle{plain}
\bibliography{Ref}

\end{document}